\documentclass[preprint,12pt,authoryear]{elsarticle}

\usepackage{amsmath}
\usepackage{newtxtext,newtxmath}
\usepackage{graphicx}
\usepackage{booktabs}
\usepackage{algorithm}
\usepackage{enumitem}
\usepackage[hidelinks]{hyperref}

\newcommand{\Vtx}{\boldsymbol{x}}
\newcommand{\Vtt}{\boldsymbol{t}}
\newcommand{\VtL}{\boldsymbol{L}}
\newcommand{\VtF}{\boldsymbol{F}}
\newcommand{\Vte}{\boldsymbol{e}}
\newcommand{\Vth}{\boldsymbol{h}}
\newcommand{\VtH}{\boldsymbol{H}}
\newcommand{\Vts}{\boldsymbol{s}}
\newcommand{\Vty}{\boldsymbol{y}}
\newcommand{\Vtd}{\boldsymbol{d}}
\newcommand{\Vteta}{\boldsymbol{\eta}}
\newcommand{\ClN}{\mathcal{N}}
\newcommand{\ClA}{\mathcal{A}}
\newcommand{\ClW}{\mathcal{W}}

\newtheorem{lemma}{Lemma}

\newcommand{\hd}[1]{\begin{tabular}[b]{@{}c@{}}#1\end{tabular}}

\newlist{steps}{enumerate}{2}
\setlist[steps,1]{label=\arabic*.,leftmargin=*,nosep}
\setlist[steps,2]{label=\arabic{stepsi}.\arabic*.,leftmargin=*,nosep}

\makeatletter
\def\ps@pprintTitle{\let\@oddhead\@empty\let\@evenhead\@empty
  \def\@oddfoot{\hfil\thepage\hfil}\let\@evenfoot\@oddfoot}
\makeatother

\begin{document}

\begin{frontmatter}

\title{A Link-Based Algorithm for Stochastic User Equilibrium with Bounded Choice and Local Detour Threshold}

\author{Takeshi Nagae}
\ead{nagae@tottori-u.ac.jp}
\affiliation{organization={Department of Social Systems and Civil Engineering, Faculty of Engineering, Tottori University},
            city={Tottori},
            country={Japan}}

\begin{abstract}
This study develops a novel solution method for stochastic user equilibrium (SUE) under the bounded choice model with local detour threshold (BCM-LDT SUE; \citealp{Rasmussen:2024aa}), which restricts the choice set in line with observed route choice behavior. The model assigns zero probability to paths whose total cost exceeds a bound or whose local detour exceeds a threshold; thus, the path choice set varies with link costs. The proposed algorithm updates link flows while implicitly enumerating the path choice set via a depth-first search (DFS) pruned with the BCM-LDT bounds. This distinguishes it from the original algorithm, which updates path flows, and enables quasi-Newton-type acceleration. Since paths are never stored, the method can handle choice sets totaling on the order of $10^7$ paths. On a personal computer, it reached equilibrium in a few to about fifteen minutes, not only for the Anaheim network under the original settings but also for larger networks such as Winnipeg and Barcelona.
\end{abstract}


\end{frontmatter}

\section{Introduction}\label{sec:intro}

This study develops a solution method for the stochastic user equilibrium (SUE) under the bounded choice model with local detour threshold (BCM-LDT; \citealp{Rasmussen:2024aa}). Consider a network represented by a directed graph $G=(\ClN,\ClA)$ with a set of nodes $\ClN$ and a set of links $\ClA$, and a set of OD pairs $\ClW$. For each $w\in\ClW$, let $K_w$ be the set of simple paths and $q_w$ the travel demand. Link costs $\Vtt(\Vtx)=(t_a(\Vtx):a\in\ClA)$ are continuous functions of link flows $\Vtx=(x_a:a\in\ClA)$. A link $a\in\ClA$ from node $i$ to node $j$ is also denoted by $(i,j)$, and its cost by $t_a$ or $t_{ij}$. The cost of path $k$ is $C_k=\sum_{a\in k}t_a$. The local detouredness of path $k$ is defined as
\begin{equation}
  \phi_k=\max_{(i,j)\in S_k}\frac{C_k(i\to j)-d(i,j)}{d(i,j)}.\label{eq:ldt}
\end{equation}
Here, $S_k$ is the set of ordered node pairs on $k$, $C_k(i\to j)$ is the cost of the subpath of $k$ from $i$ to $j$, and $d(i,j)$ is the minimum cost from node $i$ to node $j$.\footnote{\citet{Rasmussen:2024aa} define $d(i,j)$ as the minimum cost among the segment alternatives from $i$ to $j$, which coincides with the minimum node-to-node cost for positive link costs; they also use the latter in their algorithm.} The BCM-LDT choice probability of path $k\in K_w$ is
\begin{equation}
  P_k=\frac{(e^{-\theta_1(C_k-\tau d_w)}-1)_+(e^{-\theta_2(\phi_k-\gamma)}-1)_+}{\sum_{k'\in K_w}(e^{-\theta_1(C_{k'}-\tau d_w)}-1)_+(e^{-\theta_2(\phi_{k'}-\gamma)}-1)_+},\label{eq:prob}
\end{equation}
where $d_w$ is $d(i,j)$ for the origin and destination of $w$ and $(\cdot)_+=\max(\cdot,0)$. In Eq.~\eqref{eq:prob}, $\tau>1$ is the global cost bound, $\gamma>0$ the local detour threshold, and $\theta_1,\theta_2>0$ the scaling parameters for travel cost and local detouredness. A path receives zero probability if its cost reaches $\tau d_w$ or its local detouredness reaches $\gamma$; thus, the path choice set varies with link costs. \citet{Rasmussen:2024aa} formulated BCM-LDT SUE as a fixed-point problem in path flows $\Vth$:
\begin{equation}
  \Vth=\VtH(\Vth),\qquad H_k(\Vth)=q_wP_k\bigl(\Vtt(\Delta\Vth)\bigr)\quad(k\in K_w,\ w\in\ClW),\label{eq:fp-path}
\end{equation}
where $K=\bigcup_{w\in\ClW}K_w$ and $\Delta=[\delta_{ak}]_{a\in\ClA,\,k\in K}$ is the link-path incidence matrix ($\delta_{ak}=1$ if path $k$ uses link $a$ and $0$ otherwise). They also proposed an algorithm that enumerates paths by a pruned depth-first search (pruned DFS) and updates path flows at each iteration.

In this study, we first rewrite the fixed-point problem in terms of link flows $\Vtx=\Delta\Vth$. BCM-LDT SUE becomes
\begin{equation}
  \Vtx=\VtF(\Vtx)=\VtL\bigl(\Vtt(\Vtx)\bigr),\qquad L_a(\Vtt)=\sum_{w\in\ClW}\sum_{k\in K_w}\delta_{ak}\,q_wP_k(\Vtt)\quad(a\in\ClA),\label{eq:fp-link}
\end{equation}
as $\VtH$ depends on $\Vth$ only through $\Vtt(\Delta\Vth)=\Vtt(\Vtx)$. We call $\VtL$ the BCM-LDT loading.
We then extend the pruned DFS above into an algorithm that computes $\VtL(\Vtt)$ exactly with implicit path enumeration: when each path is found, its choice weight, i.e.\ the numerator of its choice probability, is accumulated on the links of the path, and the path itself is then discarded. Since the formulation and the iteration are closed in the space of link flows $\Vtx$, BCM-LDT SUE can be treated as a standard $|\ClA|$-dimensional nonlinear fixed-point problem, and can thus be solved by Anderson acceleration, a well-known quasi-Newton-type solution method. This distinguishes the present method from that of \citet{Rasmussen:2024aa}, in which the path choice set, and hence the dimension and components of the unknowns, may change from iteration to iteration; it updates the solution by the method of successive weighted averages (MSWA; \citealp{Liu:2009aa}).

To maintain a behaviorally realistic and computationally tractable path choice set, SUE models that endogenously assign zero probability to some paths have been studied. One line of research imposes bounds on path attributes: \citet{Watling:2018aa} proposed SUE based on the bounded choice model (BCM), which has been extended to the bounded path size model \citep{Duncan:2022aa}, to BCM-LDT \citep{Rasmussen:2024aa}, and to a model combining the two \citep{Duncan:2025aa}. Another line is the perturbed utility Markovian traffic equilibrium (PUME; \citealp{yao2026perturbedutilitymarkoviantraffic}), in which zero probabilities arise from the regularization of link choices at each node rather than from path-level criteria such as total cost or local detouredness. These and other SUE models formulated as fixed-point problems are commonly solved by the method of successive averages (MSA; \citealp{Sheffi:1982aa}) or the MSWA.

The rest of this paper is organized as follows. Section~\ref{sec:loading} presents the link-based loading, which extends the pruned DFS to compute link flows for given link costs without storing paths. Section~\ref{sec:anderson} describes the solution method for BCM-LDT SUE based on Anderson acceleration. Section~\ref{sec:exp} applies the method to the Anaheim network and shows that it reproduces the path choice sets reported by \citet{Rasmussen:2024aa} and reaches equilibrium within minutes on a personal computer. Section~\ref{sec:concl} concludes the paper with a discussion of limitations.

\section{Link-based loading with implicit path enumeration}\label{sec:loading}

Suppose that the link costs $\Vtt=(t_{ij}:(i,j)\in\ClA)$ and the resulting minimum costs between all node pairs, $D=[d(i,j)]$, are given. For each OD pair $w=(r,s)$, the pruned DFS of \citet{Rasmussen:2024aa} extends a subpath (prefix) from the origin $r$ in a depth-first manner; when a node $i_n$ is to be appended to a prefix $(r=i_1,\dots,i_{n-1})$, it prunes the branch if any of the following criteria is met:
\begin{itemize}[nosep]
  \item[(C1)] $i_n$ is already on the prefix;
  \item[(C2)] $C(r\to i_n)+d(i_n,s)>\tau d_w$, where $C(i_l\to i_m)=\sum_{u=l}^{m-1}t_{i_ui_{u+1}}$ is the cost of the subpath from $i_l$ to $i_m$;
  \item[(C3)] $\phi(r\to i_n)>\gamma$, where
  \begin{equation*}
    \phi(r\to i_n)=\max_{1\le l<m\le n}\frac{C(i_l\to i_m)-d(i_l,i_m)}{d(i_l,i_m)}
  \end{equation*}
  is the local detouredness of the subpath from $r$ to $i_n$.
\end{itemize}
Otherwise, $i_n$ is appended to the prefix; if $i_n=s$, a path has been found and the branch ends there. When a path $k$ is found, $C_k=C(r\to s)$ and $\phi_k=\phi(r\to s)$, so that the numerator of its choice probability, hereafter called the weight of path $k$,
\begin{equation}
  \pi_k=\bigl(e^{-\theta_1(C_k-\tau d_w)}-1\bigr)_+\bigl(e^{-\theta_2(\phi_k-\gamma)}-1\bigr)_+,\label{eq:weight}
\end{equation}
can be computed. This pruned DFS enumerates all paths with $\pi_k>0$ \citep{Rasmussen:2024aa},\footnote{For completeness, a proof is given in \ref{app:proof}.} and the choice probability is given by
\begin{equation}
  P_k=\frac{\pi_k}{\sum_{k'\in K_w}\pi_{k'}}.\label{eq:prob-weight}
\end{equation}

The pruned DFS can be extended to compute the BCM-LDT loading $\VtL(\Vtt)$ without storing paths. Substituting Eq.~\eqref{eq:prob-weight} into Eq.~\eqref{eq:fp-link} yields
\begin{equation}
  L_a(\Vtt)=\sum_{w\in\ClW}q_w\frac{\sum_{k\in K_w}\delta_{ak}\pi_k}{\sum_{k\in K_w}\pi_k}=\sum_{w\in\ClW}q_w\frac{\Pi^w_a}{\Pi^w}\quad(a\in\ClA),\label{eq:loading}
\end{equation}
where $\Pi^w$ is the total weight of the paths of OD pair $w$ and $\Pi^w_a$ is that of the paths using link $a$. These sums can be accumulated during the pruned DFS: each time a path $k\in K_w$ is found, $\pi_k$ is added to $\Pi^w$ and to $\Pi^w_a$ for each link $a$ on $k$, and the path is then discarded.\footnote{To avoid overflow, our implementation divides $\pi_k$ by a positive constant common to all paths of the same OD pair; this does not change $P_k$.}

This accumulation can be implemented within the pruned DFS written as a recursive function (Algorithm~\ref{alg:dfs}), which is called for each OD pair in the BCM-LDT loading (Algorithm~\ref{alg:loading}). PrunedDFS receives a prefix $(i_1,\dots,i_n)$ together with its cost $C=C(r\to i_n)$ and local detouredness $\phi=\phi(r\to i_n)$; it searches the extensions of the prefix while pruning branches by (C1)--(C3), and returns the total weight of the paths sharing the prefix. The value returned for an extended prefix $(i_1,\dots,i_n,j)$, i.e., the total weight of the paths sharing it, is added to $\Pi^w_{i_nj}$ (Step~3.5); hence, when the search ends, $\Pi^w_a$ and the value returned for the prefix $(r)$ are equal to $\Pi^w_a$ and $\Pi^w$ in Eq.~\eqref{eq:loading}, respectively. Since no path is stored, Algorithm~\ref{alg:loading} requires memory only for $\Pi^w_a$ ($|\ClA|$ values), a stack proportional to the recursion depth, which is less than the number of nodes, and $D$ (the square of the number of nodes). Note that the computations for different OD pairs are independent and can be executed in parallel.

\begin{algorithm}[t]
\caption{BCM-LDT loading}\label{alg:loading}
\textbf{Input:} $\Vtt$, $D$. \textbf{Output:} $\VtL(\Vtt)$.
\begin{steps}
  \item $\VtL:=\mathbf{0}$.
  \item For each $w=(r,s)\in\ClW$:
  \begin{steps}
    \item $\Pi^w_a:=0$ $(a\in\ClA)$.
    \item $\Pi^w:=\mathrm{PrunedDFS}(r;\ 0,\ 0)$.
    \item $L_a:=L_a+q_w\Pi^w_a/\Pi^w$ $(a\in\ClA)$.
  \end{steps}
  \item Return $\VtL$.
\end{steps}
\end{algorithm}

\begingroup
\renewcommand{\thealgorithm}{1.1}
\begin{algorithm}[t]
\caption{PrunedDFS$(i_1,\dots,i_n;\ C,\ \phi)$}\label{alg:dfs}
\textbf{Input:} a prefix $(i_1,\dots,i_n)$, its cost $C=C(r\to i_n)$, and its local detouredness $\phi=\phi(r\to i_n)$.\\
\textbf{Global:} $\Vtt$, $D$, $\theta_1,\theta_2,\tau,\gamma$, $w=(r,s)$; updates $\Pi^w_{ij}$.\\
\textbf{Output:} the total weight of the paths sharing the prefix $(i_1,\dots,i_n)$.
\begin{steps}
  \item If $i_n=s$, return $\pi_k$ in Eq.~\eqref{eq:weight} with $C_k=C$ and $\phi_k=\phi$.
  \item $\sigma:=0$.
  \item For each link $(i_n,j)\in\ClA$:
  \begin{steps}
    \item $C':=C+t_{i_nj}$.
    \item If $j\in\{i_1,\dots,i_n\}$ (C1) or $C'+d(j,s)>\tau d_w$ (C2), skip to the next link.
    \item $\phi':=\max\{\phi,\ \max_{l\le n}(C(i_l\to j)-d(i_l,j))/d(i_l,j)\}$; if $\phi'>\gamma$ (C3), skip to the next link.
    \item $\sigma':=\mathrm{PrunedDFS}(i_1,\dots,i_n,j;\ C',\ \phi')$.
    \item $\Pi^w_{i_nj}:=\Pi^w_{i_nj}+\sigma'$; $\sigma:=\sigma+\sigma'$.
  \end{steps}
  \item Return $\sigma$.
\end{steps}
\end{algorithm}
\endgroup
\addtocounter{algorithm}{-1}

\section{Fixed-point iteration with Anderson acceleration}\label{sec:anderson}

Let $\Vtx_n$ denote the link flows at the $n$th iteration $(n=0,1,\dots)$ and $\Vte_n=\VtF(\Vtx_n)-\Vtx_n$ the residual. Anderson acceleration \citep{Anderson:1965aa,Walker:2011aa} is a quasi-Newton-type method \citep{Fang:2009aa} that constructs the direction $\Vtd_n$ from the $m$ most recent differences $\Vts_i=\Vtx_{i+1}-\Vtx_i$ and $\Vty_i=\Vte_{i+1}-\Vte_i$ $(i=n-m,\dots,n-1)$.\footnote{MSA and MSWA iterate $\Vtx_{n+1}=\Vtx_n+\lambda_n\Vte_n$; the direction is fixed to the residual $\Vte_n$, and they differ only in the choice of the step size $\lambda_n$.} The direction is computed as
\begin{equation}
  \Vtd_n=\omega\Vte_n-\sum_{i=n-m}^{n-1}\eta_{n,i}(\Vts_i+\omega\Vty_i),\label{eq:anderson}
\end{equation}
where $\omega\in(0,1]$ is a damping parameter and $\Vteta_n=(\eta_{n,i})$ is the solution of the least-squares problem $\min_{\Vteta}\|\Vte_n-\sum_{i=n-m}^{n-1}\eta_i\Vty_i\|$. Note that (i) the lower limit of the sums is $\max(0,n-m)$, so that $\Vtd_0=\omega\Vte_0$; (ii) since $\Vteta_n$ is $m$-dimensional, computing $\Vtd_n$ requires negligible time and memory compared with evaluating $\VtF$.

The whole procedure for solving BCM-LDT SUE can be written as Algorithm~\ref{alg:sue}.

\begin{algorithm}[t]
\caption{Solution of BCM-LDT SUE}\label{alg:sue}
\textbf{Input:} network $G=(\ClN,\ClA)$, OD demands $q_w$ $(w\in\ClW)$, link cost functions $\Vtt(\cdot)$, model parameters $\theta_1,\theta_2,\tau,\gamma$, and algorithm parameters $m,\omega,\varepsilon$.\\
\textbf{Output:} equilibrium link flows $\Vtx$.
\begin{description}[nosep,font=\normalfont]
  \item[Step 0 (Initialization).] $\Vtx_0:=\VtF(\mathbf{0})$, computed by the procedure in Step~1; $n:=0$.
  \item[Step 1 (Evaluation).] Compute $\Vtt(\Vtx_n)$ and the corresponding minimum cost matrix $D=[d(i,j)]$. Compute $\VtF(\Vtx_n)=\VtL(\Vtt(\Vtx_n))$ by Algorithm~\ref{alg:loading}. $\Vte_n:=\VtF(\Vtx_n)-\Vtx_n$.\textsuperscript{a}
  \item[Step 2 (Convergence test).] If $\|\Vte_n\|/\|\Vtx_n\|<\varepsilon$, return $\Vtx_n$.
  \item[Step 3 (Direction finding).] If $n\ge1$, $\Vts_{n-1}:=\Vtx_n-\Vtx_{n-1}$ and $\Vty_{n-1}:=\Vte_n-\Vte_{n-1}$. Compute $\Vtd_n$ by Eq.~\eqref{eq:anderson}.
  \item[Step 4 (Move).] $\Vtx_{n+1}:=\Vtx_n+\Vtd_n$; $n:=n+1$; go to Step~1.\textsuperscript{b}
\end{description}
\vspace{1ex}
\footnotesize
\textsuperscript{a}\,Strictly speaking, $\Vtx_n$ may have negative components. Our implementation then evaluates link costs at $[\Vtx_n]_+$, the componentwise positive part, i.e., $\VtF(\Vtx)=\VtL(\Vtt([\Vtx]_+))$, and returns $\VtF(\Vtx_n)$ instead of $\Vtx_n$ in Step~2.\\
\textsuperscript{b}\,Since $\VtF$ is not guaranteed to be a contraction mapping, the residual norm $\|\Vte_n\|$ does not necessarily decrease at every iteration of Algorithm~\ref{alg:sue}. Our implementation then halves the step size until the residual norm decreases, and restarts Anderson acceleration by discarding the stored differences if this fails. These safeguards (a and b) are not essential to the proposed method and are not discussed further; neither was activated in the numerical experiments in Section~\ref{sec:exp}.
\end{algorithm}

\section{Numerical experiments}\label{sec:exp}

This section examines the proposed method on the Anaheim network in terms of the consistency with the results of \citet{Rasmussen:2024aa} and the computational efficiency.\footnote{Although comparing equilibrium solutions across route choice models is beyond the main scope of this paper, interested readers may refer to \ref{app:sf}, where BCM-LDT is compared with MNL, the cost-based BCM, and a Markovian logit model on the Sioux Falls network.} Following \citet{Rasmussen:2024aa}, we use the following data and settings:
\begin{itemize}[nosep]
  \item The network data, consisting of 416 nodes, 914 links, and 1,406 OD pairs, are taken from the same source \citep{TNTP:2026}.
  \item $\theta_1=\theta_2=0.2$ and $\tau=1.6$; the values of $\gamma$ are given below.
  \item The cost of each link is its travel time (in minutes), given by the BPR function with the parameters in the data, plus 0.5 times its length (in kilometers).\footnote{Link lengths in the data are given in feet and converted with 0.3/1000 km/ft, as in the implementation of \citet{Rasmussen:2024aa}.}
  \item Links connected to OD nodes can be used only by trips with an origin or destination at these nodes; that is, paths do not pass through OD nodes, and $d(i,j)$ is computed accordingly.
\end{itemize}
In Algorithm~\ref{alg:sue}, we set $m=6$, $\omega=0.2$, and $\varepsilon=10^{-9}$, and compute $D$ by Dijkstra's algorithm. All computations were performed with Julia 1.13 on a Mac mini (Apple M4, 32 GB memory) using a single thread.

We first compare the path choice sets obtained by the proposed method with those of \citet{Rasmussen:2024aa}. As in their settings, we fix the link travel times at the UE values provided in the data and count the paths with positive weight (i.e., the choice set) for OD pair (4,7) with $\tau=1.6$ and $\gamma=0.1$, 0.2, 0.4, 0.6, 0.8, and 1.0. The resulting numbers, 1, 1, 6, 33, 600, and 3,725, are of the same order of magnitude as those in their Fig.~14. In addition, at the equilibrium with $\gamma=0.8$, the choice set of OD pair (4,7) consists of 351 paths, which is equal to the number they reported.

Next, we solve BCM-LDT SUE by Algorithm~\ref{alg:sue} for $\gamma=0.6$, 0.8, and 1.0. Table~\ref{tab:anaheim} shows the number of iterations, the computation time (including initialization), and the total choice set size over all OD pairs at equilibrium. Fig.~\ref{fig:anaheim} shows the relative residual $\|\Vte_n\|/\|\Vtx_n\|$ against the iteration number $n$, where $n=0$ corresponds to the initial solution. For all values of $\gamma$, the residual decreases almost geometrically at a similar rate, by about one order of magnitude every two to three iterations, and the convergence criterion $\|\Vte_n\|/\|\Vtx_n\|<\varepsilon=10^{-9}$ (dotted line in Fig.~\ref{fig:anaheim}) is met within 24 iterations. While the number of iterations hardly depends on $\gamma$, the computation time increases with the choice set size. About 99\% or more of the computation time is spent in Algorithm~\ref{alg:loading} (Table~\ref{tab:anaheim}). The computation time is the largest for $\gamma=1.0$ but is still less than 3 minutes (156~s).\footnote{\citet{Rasmussen:2024aa} reported that each iteration of their algorithm took on average 1.5, 8, and 180 minutes for $\gamma=0.6$, 0.8, and 1.0, respectively.} The total choice set size at equilibrium reaches about 12 million for $\gamma=1.0$, but the proposed method stores none of these paths. Although the details are omitted for brevity, parallelizing the loading over OD pairs reduces the computation time further; with 10 threads, it was reduced to about one fifth in our environment.

\begin{table}[t]
\centering
\caption{Results for the Anaheim network.}\label{tab:anaheim}
\begin{tabular}{rrrrr}
\toprule
$\gamma$ & Iterations & Time (s) & Share of Algorithm~\ref{alg:loading} (\%) & Choice set size\\
\midrule
0.6 & 21 & 21 & 98.8 & 969,111\\
0.8 & 24 & 69 & 99.8 & 3,788,881\\
1.0 & 23 & 156 & 99.9 & 11,617,998\\
\bottomrule
\end{tabular}
\end{table}

\begin{figure}[t]
\centering
\includegraphics[width=0.7\linewidth]{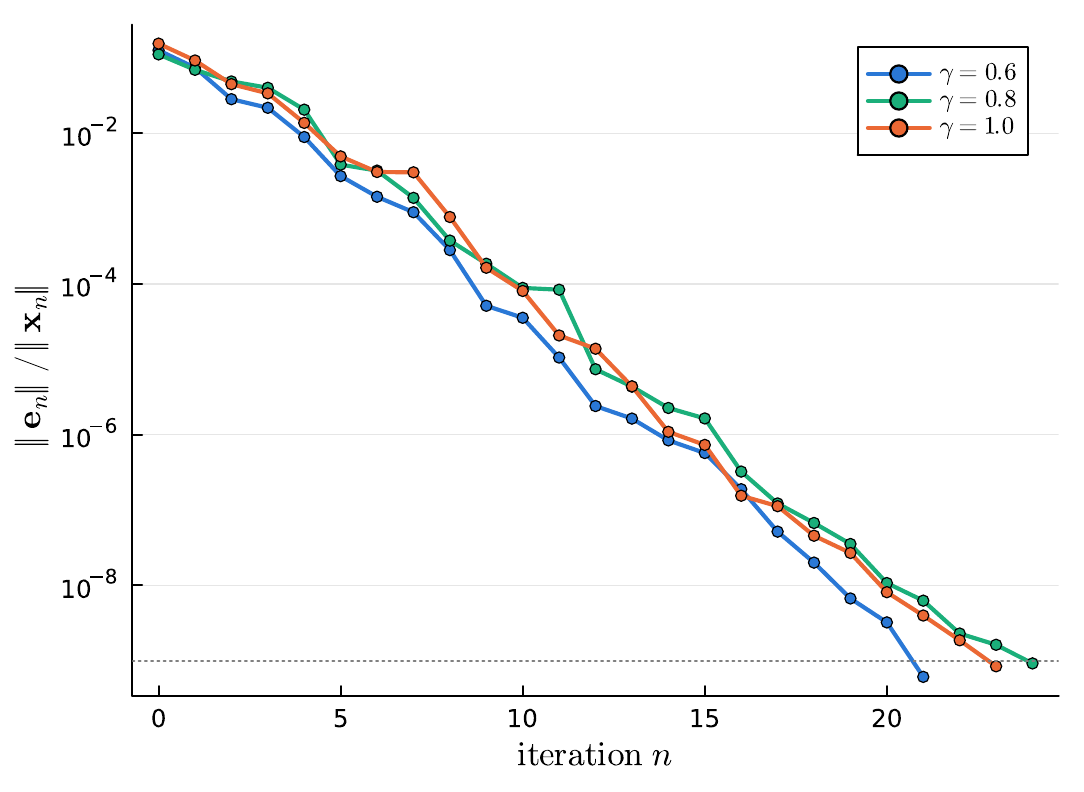}
\caption{Convergence of the relative residual $\|\Vte_n\|/\|\Vtx_n\|$.}\label{fig:anaheim}
\end{figure}

\section{Conclusion}\label{sec:concl}

This study developed a novel solution method for BCM-LDT SUE \citep{Rasmussen:2024aa}. We first reformulated BCM-LDT SUE as a fixed-point problem in link flows, $\Vtx=\VtF(\Vtx)=\VtL(\Vtt(\Vtx))$. We then extended their pruned DFS into an algorithm that computes the BCM-LDT loading $\VtL(\Vtt)$ without storing paths. Since the unknowns are only the link flows, BCM-LDT SUE can be solved by Anderson acceleration, a quasi-Newton-type method. On the Anaheim network, the proposed method yielded path choice sets consistent with those of \citet{Rasmussen:2024aa} and reached equilibrium in less than 3 minutes using a single thread; the choice sets contained about 12 million paths in total, none of which were stored. Additional experiments on the Winnipeg, Chicago Sketch, and Barcelona networks showed that the proposed method is practical when the total choice set size is up to the order of $10^7$; for Winnipeg and Barcelona, equilibrium was reached within 15 minutes (\ref{app:mid}). The code is available at \url{https://github.com/nagae/BCMLDT}.

Several issues remain for future research. First, the proposed method is not guaranteed to converge globally. On the side of the problem, the uniqueness of BCM-LDT SUE has not been established \citep{Rasmussen:2024aa}. On the side of the method, the convergence theory of Anderson acceleration is local and assumes that the fixed-point map is a contraction mapping \citep{Toth:2015aa}, which is not guaranteed for $\VtF$. Second, on large-scale networks such as Chicago Regional, the bounds of BCM-LDT may fail to restrict the choice sets to a manageable size: even with $\tau=1.02$ and $\gamma=0.02$, they contain more than 1,000 paths per OD pair on average, presumably because the network contains a huge number of paths with nearly equal costs (\ref{app:chicago}). Developing choice set criteria that are both behaviorally realistic and computationally tractable on such networks remains an open challenge.

\section*{Funding}
No funding was obtained for this work.

\section*{Disclosure statement}
The author reports there are no competing interests to declare.

\section*{Declaration of generative AI use}
During the preparation of this work, the author used Claude (Anthropic) for coding and writing assistance. After using this tool, the author reviewed and edited the content as needed and takes full responsibility for the content of the published article.

\section*{Data availability statement}
The network data used in this study are publicly available from the Transportation Networks for Research repository \citep{TNTP:2026}. The code implementing the proposed method, together with scripts that reproduce the results in Section~\ref{sec:exp}, is available at \url{https://github.com/nagae/BCMLDT}.

\appendix

\section{Exactness of the pruned DFS}\label{app:proof}

\begin{lemma}
Suppose that all link costs are positive. Then, for each OD pair $w$, the paths found by the pruned DFS with (C1)--(C3) include every path $k\in K_w$ with $\pi_k>0$ exactly once.
\end{lemma}

\noindent\textit{Proof.} By Eq.~\eqref{eq:weight}, $\pi_k>0$ if and only if $C_k<\tau d_w$ and $\phi_k<\gamma$. Let $k=(r=i_1,\dots,i_N=s)\in K_w$ be a path with $\pi_k>0$. Since the DFS visits exactly once every prefix that is not pruned, it suffices to show that none of the prefixes $(i_1,\dots,i_n)$ of $k$ is pruned. (C1) is not met because $k\in K_w$. (C2) is not met because $d(i_n,s)\le C(i_n\to s)$, so that $C(r\to i_n)+d(i_n,s)\le C_k<\tau d_w$. (C3) is not met because $\phi(r\to i_n)$ is the maximum over a subset of the node pairs in $S_k$, so that $\phi(r\to i_n)\le\phi_k<\gamma$. \hfill$\square$

\section{Additional numerical experiments}\label{app:exp}
\setcounter{table}{0}\setcounter{figure}{0}

This appendix reports three additional numerical experiments that complement those in Section~\ref{sec:exp}. In \ref{app:sf}, we compare BCM-LDT with other route choice models on the Sioux Falls network, for which all paths can be enumerated, and show that BCM-LDT excludes paths with excessive detours and keeps the choice sets reasonably small. In \ref{app:mid}, we use mid-size networks to show that the proposed method obtains equilibrium solutions within practical computation time as long as the total choice set size is up to the order of $10^7$. In \ref{app:chicago}, we show that, on large-scale networks, the BCM-LDT criteria may fail to restrict the choice sets to a manageable size.

\subsection{Comparison with other route choice models: Sioux Falls}\label{app:sf}

We examine how BCM-LDT keeps the choice sets reasonably small by excluding paths with excessive detours, comparing it with other route choice models on the Sioux Falls network, for which all paths can be enumerated. We assume that link costs consist of travel times given by BPR-type functions of link flows, and evaluate the distribution of path costs at the respective equilibrium costs.\footnote{The SUEs of all models were obtained by Algorithm~\ref{alg:sue} with $\VtF$ replaced by the loading function of each model, using the same parameters as in Section~\ref{sec:exp} except $\varepsilon=10^{-10}$; hence, the differences in the choice sets shown below do not stem from the accuracy of the solutions.} The scaling parameter for travel cost is set to 0.5 in all models ($\theta_1=0.5$ in BCM-LDT), and the other parameters of BCM-LDT and cost-based BCM are set to the base values in \citet{Rasmussen:2024aa}. MNL is defined over all acyclic paths, and the Markovian logit model \citep{Bell:1995aa,Akamatsu:1996aa,Akamatsu:1997aa} over all paths including cyclic ones. Table~\ref{tab:sf} shows the total choice set sizes, i.e., the numbers of paths with positive choice probabilities, at the equilibria. The total choice set size of BCM-LDT is about two thirds of that of cost-based BCM, since BCM-LDT further excludes the paths with large local detours.

\begin{table}[t]
\centering
\caption{Models compared and their total choice set sizes at equilibrium on the Sioux Falls network.}\label{tab:sf}
\begin{tabular}{llr}
\toprule
Model & Parameters & Choice set size\\
\midrule
BCM-LDT & $\theta_2=0.2$, $\tau=1.6$, $\gamma=0.8$ & 3,426\\
Cost-based BCM & $\tau=1.6$ & 5,153\\
MNL & -- & 1,632,820\\
Markovian logit & -- & $\infty$\\
\bottomrule
\end{tabular}
\end{table}

For each path $k\in K_w$ of each OD pair $w$, we define the relative cost $R_k=C_k/d_w$, and the share of trips using paths with relative costs above $r$ as
\begin{equation}
  S(r)=\frac{\sum_{w}\sum_{k\in K_w:\,R_k>r}h_k}{\sum_{w}q_w}.
\end{equation}
Roughly speaking, $S(r)$ is the share of travelers who take detours costing more than $r$ times the minimum. Fig.~\ref{fig:sf} shows $S(r)$ for each model. For any $r$, $S(r)$ of the Markovian logit model is larger than that of MNL, which is larger than those of cost-based BCM and BCM-LDT. In cost-based BCM and BCM-LDT, $S(r)=0$ for $r\ge\tau=1.6$; in contrast, MNL and the Markovian logit model assign flows even to paths costing more than 10 times the minimum.

\begin{figure}[t]
\centering
\includegraphics[width=0.7\linewidth]{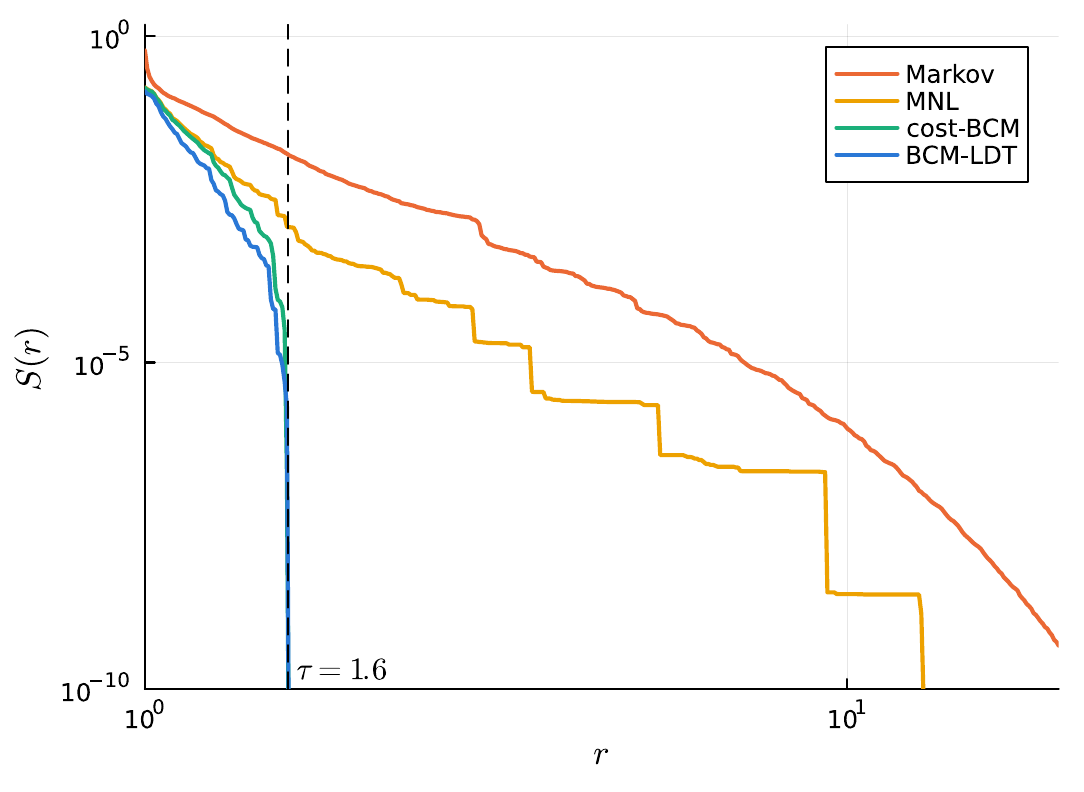}
\caption{Share of travelers taking detours costing more than $r$ times the minimum, $S(r)$, at the equilibria of four route choice models. The dashed line indicates the global cost bound $\tau=1.6$ of cost-based BCM and BCM-LDT.}\label{fig:sf}
\end{figure}

\subsection{Applicability of the proposed method: Winnipeg, Chicago Sketch, and Barcelona}\label{app:mid}

We examine the applicability of the proposed method on three mid-size networks: Winnipeg, Chicago Sketch, and Barcelona. The data are taken from the same source as in Section~\ref{sec:exp}. We assume that link costs consist of travel times only, and set $\theta_1=\theta_2=0.2$, the global cost bound $\tau=1.2$, and the local detour threshold $\gamma=0.2$ or 0.4. The parameters of Algorithm~\ref{alg:sue} are the same as in Section~\ref{sec:exp}, and all computations were performed using a single thread.

We first evaluated the choice set sizes and the computation time of Algorithm~\ref{alg:loading} by executing it once for all OD pairs with link travel times fixed at the UE values provided in the data (Table~\ref{tab:mid}). When the total choice set size was up to the order of $10^7$, one execution took about 30~s or less. For Chicago Sketch with $\gamma=0.4$, in contrast, the total choice set size reached $2.7\times10^9$ and one execution took 18 minutes; a single OD pair alone had $3.3\times10^8$ paths.

We then solved the equilibria for Winnipeg ($\gamma=0.4$) and Barcelona ($\gamma=0.2$) by Algorithm~\ref{alg:sue}, with travel times given by BPR-type functions of link flows.\footnote{Chicago Sketch was excluded because it contains links with zero free-flow travel time, for which some minimum costs become zero and the local detouredness cannot be defined.} In both cases, the algorithm converged within 15 minutes, the total choice set size at equilibrium remained of the order of $10^6$, and Algorithm~\ref{alg:loading} accounted for more than 99\% of the computation time (Table~\ref{tab:mid}).

\begin{table}[t]
\centering
\caption{Choice set sizes and computation times on mid-size networks.}\label{tab:mid}
\small\setlength{\tabcolsep}{4pt}
\begin{tabular}{lcrrrrr}
\toprule
& & \multicolumn{2}{c}{Fixed costs (Alg.~\ref{alg:loading})} & \multicolumn{3}{c}{SUE (Alg.~\ref{alg:sue})}\\
\cmidrule(lr){3-4}\cmidrule(lr){5-7}
Network & $\gamma$ & Time (s) & \hd{Choice\\ set size} & Iterations & Time (s) & \hd{Choice\\ set size}\\
\midrule
Winnipeg                 & 0.2 & 1.3 & $6.8\times10^5$ & -- & -- & --\\
(1,052 / 2,836 / 4,344)  & 0.4 & 21.1 & $2.2\times10^7$ & 46 & 261 & $5.6\times10^6$\\
\midrule
Chicago Sketch           & 0.2 & 27.5 & $4.2\times10^7$ & -- & -- & --\\
(933 / 2,950 / 93,135)   & 0.4 & 1,087 & $2.7\times10^9$ & -- & -- & --\\
\midrule
Barcelona                & 0.2 & 19.4 & $8.6\times10^6$ & 49 & 652 & $7.0\times10^6$\\
(1,020 / 2,522 / 7,922)  & 0.4 & -- & -- & -- & -- & --\\
\bottomrule
\end{tabular}
\par\smallskip
\begin{minipage}{\linewidth}\footnotesize
Note: Numbers in parentheses are the numbers of nodes, links, and OD pairs. Fixed costs: link travel times fixed at the UE values provided in the data. SUE: travel times given by BPR-type functions of link flows. Barcelona with $\gamma=0.4$ was not executed, because a preliminary count on randomly sampled OD pairs indicated that the total choice set size would exceed the order of $10^8$.
\end{minipage}
\end{table}

\subsection{Limits of choice set restriction by the BCM-LDT criteria: Chicago Regional}\label{app:chicago}

We show that, on large-scale networks, the BCM-LDT criteria may fail to restrict the choice sets to a manageable size, using the Chicago Regional network (12,982 nodes, 39,018 links, and 2,296,227 OD pairs) as an example. As in \ref{app:mid}, we assume that link costs consist of travel times only, and fix them at the UE values provided in the data. Since the total choice set size far exceeds a manageable size, as shown below, we counted the paths for 200 randomly sampled OD pairs under four combinations of the global cost bound $\tau$ and the local detour threshold $\gamma$ (Table~\ref{tab:chicago}).

Even with $\tau=1.02$ and $\gamma=0.02$, 1,168.6 paths per OD pair remain on average, and up to 201,714 paths for a single OD pair. The choice sets shrink to tens of paths per OD pair only when the local detour threshold is tightened to 1\%, and even then the total choice set size is estimated at $1.3\times10^8$, beyond the manageable range up to the order of $10^7$ found in \ref{app:mid}. This is presumably because the network contains combinatorially many apparent alternatives with nearly equal costs, owing to its dense local streets and long trips.\footnote{Without the local detour threshold, more than 10,000 and 50,000 paths per OD pair on average have costs within 1\% and 2\% of the minimum, respectively.} Since the number of paths never decreases as $\tau$ or $\gamma$ increases, the BCM-LDT criteria with bounds such as those used for Anaheim ($\tau=1.6$ and $\gamma=0.6$--1.0) can by no means restrict the choice sets to a manageable size on this network.

\begin{table}[t]
\centering
\caption{Number of paths for 200 randomly sampled OD pairs on the Chicago Regional network.}\label{tab:chicago}
\begin{tabular}{llrrr}
\toprule
$\gamma$ & $\tau$ & \hd{Mean per\\ OD pair} & \hd{Max per\\ OD pair} & \hd{Estimated total\\ choice set size}\\
\midrule
0.01 & 1.01 & 55.9 & 4,016 & $1.3\times10^8$\\
     & 1.02 & 55.9 & 4,016 & $1.3\times10^8$\\
0.02 & 1.01 & 150.8 & 9,472 & $3.5\times10^8$\\
     & 1.02 & 1,168.6 & 201,714 & $2.7\times10^9$\\
\bottomrule
\end{tabular}
\par\smallskip
\begin{minipage}{\linewidth}\footnotesize
Note: The estimated total choice set size is the mean per OD pair multiplied by the number of OD pairs (2,296,227). Since the distribution of the number of paths is heavy-tailed, it may be underestimated.
\end{minipage}
\end{table}

\bibliographystyle{bcmldt-harv}
\bibliography{bcmldt}

\begin{thebibliography}{15}
\expandafter\ifx\csname natexlab\endcsname\relax\def\natexlab#1{#1}\fi
\providecommand{\url}[1]{\texttt{#1}}
\providecommand{\href}[2]{#2}
\providecommand{\path}[1]{#1}
\providecommand{\DOIprefix}{doi:}
\providecommand{\ArXivprefix}{arXiv:}
\providecommand{\URLprefix}{URL: }
\providecommand{\Pubmedprefix}{pmid:}
\providecommand{\doi}[1]{\href{https://doi.org/#1}{\path{#1}}}
\providecommand{\Pubmed}[1]{\href{pmid:#1}{\path{#1}}}
\providecommand{\bibinfo}[2]{#2}
\ifx\xfnm\relax \def\xfnm[#1]{\unskip,\space#1}\fi
%Type = Article
\bibitem[{Akamatsu(1996)}]{Akamatsu:1996aa}
\bibinfo{author}{Akamatsu, T.}, \bibinfo{year}{1996}.
\newblock \bibinfo{title}{Cyclic flows, {M}arkov process and stochastic traffic
  assignment}.
\newblock \bibinfo{journal}{Transportation Research Part B: Methodological}
  \bibinfo{volume}{30}, \bibinfo{pages}{369--386}.
\newblock \DOIprefix\doi{10.1016/0191-2615(96)00003-3}.
%Type = Article
\bibitem[{Akamatsu(1997)}]{Akamatsu:1997aa}
\bibinfo{author}{Akamatsu, T.}, \bibinfo{year}{1997}.
\newblock \bibinfo{title}{Decomposition of path choice entropy in general
  transport networks}.
\newblock \bibinfo{journal}{Transportation Science} \bibinfo{volume}{31},
  \bibinfo{pages}{349--362}.
\newblock \DOIprefix\doi{10.1287/trsc.31.4.349}.
%Type = Article
\bibitem[{Anderson(1965)}]{Anderson:1965aa}
\bibinfo{author}{Anderson, D.G.}, \bibinfo{year}{1965}.
\newblock \bibinfo{title}{Iterative procedures for nonlinear integral
  equations}.
\newblock \bibinfo{journal}{Journal of the ACM} \bibinfo{volume}{12},
  \bibinfo{pages}{547--560}.
\newblock \DOIprefix\doi{10.1145/321296.321305}.
%Type = Article
\bibitem[{Bell(1995)}]{Bell:1995aa}
\bibinfo{author}{Bell, M.G.H.}, \bibinfo{year}{1995}.
\newblock \bibinfo{title}{Alternatives to {D}ial's logit assignment algorithm}.
\newblock \bibinfo{journal}{Transportation Research Part B: Methodological}
  \bibinfo{volume}{29}, \bibinfo{pages}{287--295}.
\newblock \DOIprefix\doi{10.1016/0191-2615(95)00005-X}.
%Type = Article
\bibitem[{Duncan et~al.(2025)Duncan, Rasmussen, Watling, Cazor and
  Nielsen}]{Duncan:2025aa}
\bibinfo{author}{Duncan, L.C.}, \bibinfo{author}{Rasmussen, T.K.},
  \bibinfo{author}{Watling, D.P.}, \bibinfo{author}{Cazor, L.},
  \bibinfo{author}{Nielsen, O.A.}, \bibinfo{year}{2025}.
\newblock \bibinfo{title}{Identification of local detouredness and correlation
  effects in a bounded route choice model: solution \& estimation in
  large-scale networks}.
\newblock \bibinfo{journal}{Transportmetrica A: Transport Science}.
  \bibinfo{note}{{In} press}.
\newblock \DOIprefix\doi{10.1080/23249935.2025.2543467}.
%Type = Article
\bibitem[{Duncan et~al.(2022)Duncan, Watling, Connors, Rasmussen and
  Nielsen}]{Duncan:2022aa}
\bibinfo{author}{Duncan, L.C.}, \bibinfo{author}{Watling, D.P.},
  \bibinfo{author}{Connors, R.D.}, \bibinfo{author}{Rasmussen, T.K.},
  \bibinfo{author}{Nielsen, O.A.}, \bibinfo{year}{2022}.
\newblock \bibinfo{title}{A bounded path size route choice model excluding
  unrealistic routes: formulation and estimation from a large-scale {GPS}
  study}.
\newblock \bibinfo{journal}{Transportmetrica A: Transport Science}
  \bibinfo{volume}{18}, \bibinfo{pages}{435--493}.
\newblock \DOIprefix\doi{10.1080/23249935.2021.1872730}.
%Type = Article
\bibitem[{Fang and Saad(2009)}]{Fang:2009aa}
\bibinfo{author}{Fang, H.R.}, \bibinfo{author}{Saad, Y.}, \bibinfo{year}{2009}.
\newblock \bibinfo{title}{Two classes of multisecant methods for nonlinear
  acceleration}.
\newblock \bibinfo{journal}{Numerical Linear Algebra with Applications}
  \bibinfo{volume}{16}, \bibinfo{pages}{197--221}.
\newblock \DOIprefix\doi{10.1002/nla.617}.
%Type = Article
\bibitem[{Liu et~al.(2009)Liu, He and He}]{Liu:2009aa}
\bibinfo{author}{Liu, H.X.}, \bibinfo{author}{He, X.}, \bibinfo{author}{He,
  B.}, \bibinfo{year}{2009}.
\newblock \bibinfo{title}{Method of successive weighted averages ({MSWA}) and
  self-regulated averaging schemes for solving stochastic user equilibrium
  problem}.
\newblock \bibinfo{journal}{Networks and Spatial Economics}
  \bibinfo{volume}{9}, \bibinfo{pages}{485--503}.
\newblock \DOIprefix\doi{10.1007/s11067-007-9023-x}.
%Type = Article
\bibitem[{Rasmussen et~al.(2024)Rasmussen, Duncan, Watling and
  Nielsen}]{Rasmussen:2024aa}
\bibinfo{author}{Rasmussen, T.K.}, \bibinfo{author}{Duncan, L.C.},
  \bibinfo{author}{Watling, D.P.}, \bibinfo{author}{Nielsen, O.A.},
  \bibinfo{year}{2024}.
\newblock \bibinfo{title}{Local detouredness: A new phenomenon for modelling
  route choice and traffic assignment}.
\newblock \bibinfo{journal}{Transportation Research Part B: Methodological}
  \bibinfo{volume}{190}, \bibinfo{pages}{103052}.
\newblock \DOIprefix\doi{10.1016/j.trb.2024.103052}.
%Type = Article
\bibitem[{Sheffi and Powell(1982)}]{Sheffi:1982aa}
\bibinfo{author}{Sheffi, Y.}, \bibinfo{author}{Powell, W.B.},
  \bibinfo{year}{1982}.
\newblock \bibinfo{title}{An algorithm for the equilibrium assignment problem
  with random link times}.
\newblock \bibinfo{journal}{Networks} \bibinfo{volume}{12},
  \bibinfo{pages}{191--207}.
\newblock \DOIprefix\doi{10.1002/net.3230120209}.
%Type = Article
\bibitem[{Toth and Kelley(2015)}]{Toth:2015aa}
\bibinfo{author}{Toth, A.}, \bibinfo{author}{Kelley, C.T.},
  \bibinfo{year}{2015}.
\newblock \bibinfo{title}{Convergence analysis for {A}nderson acceleration}.
\newblock \bibinfo{journal}{SIAM Journal on Numerical Analysis}
  \bibinfo{volume}{53}, \bibinfo{pages}{805--819}.
\newblock \DOIprefix\doi{10.1137/130919398}.
%Type = Misc
\bibitem[{{Transportation Networks for Research Core Team}(2026)}]{TNTP:2026}
\bibinfo{author}{{Transportation Networks for Research Core Team}},
  \bibinfo{year}{2026}.
\newblock \bibinfo{title}{{Transportation Networks for Research}}.
\newblock
  \bibinfo{howpublished}{\url{https://github.com/bstabler/TransportationNetworks}}.
\newblock \bibinfo{note}{Accessed October 20, 2026}.
%Type = Article
\bibitem[{Walker and Ni(2011)}]{Walker:2011aa}
\bibinfo{author}{Walker, H.F.}, \bibinfo{author}{Ni, P.}, \bibinfo{year}{2011}.
\newblock \bibinfo{title}{{A}nderson acceleration for fixed-point iterations}.
\newblock \bibinfo{journal}{SIAM Journal on Numerical Analysis}
  \bibinfo{volume}{49}, \bibinfo{pages}{1715--1735}.
\newblock \DOIprefix\doi{10.1137/10078356X}.
%Type = Article
\bibitem[{Watling et~al.(2018)Watling, Rasmussen, Prato and
  Nielsen}]{Watling:2018aa}
\bibinfo{author}{Watling, D.P.}, \bibinfo{author}{Rasmussen, T.K.},
  \bibinfo{author}{Prato, C.G.}, \bibinfo{author}{Nielsen, O.A.},
  \bibinfo{year}{2018}.
\newblock \bibinfo{title}{Stochastic user equilibrium with a bounded choice
  model}.
\newblock \bibinfo{journal}{Transportation Research Part B: Methodological}
  \bibinfo{volume}{114}, \bibinfo{pages}{254--280}.
\newblock \DOIprefix\doi{10.1016/j.trb.2018.05.004}.
%Type = Misc
\bibitem[{Yao and Zhang(2026)}]{yao2026perturbedutilitymarkoviantraffic}
\bibinfo{author}{Yao, R.}, \bibinfo{author}{Zhang, K.}, \bibinfo{year}{2026}.
\newblock \bibinfo{title}{Perturbed utility {M}arkovian traffic equilibrium:
  theory and computation}.
\newblock \href{http://arxiv.org/abs/2607.09568}{{\tt arXiv:2607.09568}}.

\end{thebibliography}

\end{document}